\documentclass[11pt,reqno]{article}

\usepackage[margin=1.15in]{geometry}
\usepackage{graphicx}
\usepackage{subcaption}
\usepackage{stmaryrd}
\usepackage{amssymb}
\usepackage{amsthm}
\usepackage{dsfont}
\usepackage{hyperref}
\usepackage{enumerate}
\usepackage{mathrsfs}
\usepackage{amsmath, dsfont, amssymb, amsthm, mathrsfs}
\usepackage{stmaryrd}
\usepackage[lite,initials,msc-links]{amsrefs}
\usepackage{amsmath}
\usepackage{centernot}
\usepackage{mathtools}
\usepackage{xcolor}
\usepackage[normalem]{ulem}
\let\counterwithout\relax

\usepackage{chngcntr}

\usepackage{comment}
\excludecomment{comment}

\usepackage{hyperref}
\hypersetup{
	colorlinks=true,
	linkcolor= black,
	citecolor = black,
	linktocpage = True,
}

\newtheorem{theorem}{Theorem}
\newtheorem{lemma}[theorem]{Lemma}
\newtheorem{proposition}[theorem]{Proposition}
\newtheorem{corollary}[theorem]{Corollary}
\numberwithin{equation}{section}
\theoremstyle{definition}

\counterwithout{equation}{section}

\theoremstyle{definition}

\theoremstyle{remark}
\newtheorem{remark}{Remark} 
\allowdisplaybreaks

\newcommand{\ZZ}{\mathbb{Z}}
\newcommand{\PP}{\mathbb{P}}

\newcommand{\E}{\mathbb{E}}
\newcommand{\id}{\mathds{1}}
\newcommand{\NN}{\mathbb{N}}
\newcommand{\RR}{\mathbb{R}}

\newcommand{\Reff}{\mathcal{R}_{\mathrm{eff}}}

\renewcommand{\bf}[1]{\mathbf{#1}}
\newcommand{\wh}[1]{\widehat{#1}}
\renewcommand{\c}[1]{\mathcal{#1}}

\DeclareMathOperator{\LE}{LE}
\DeclareMathOperator{\Gr}{\mathbf{G}}
\newcommand{\PSpine}{\mathbb{P}_{\mathcal{S}}}

\newcommand{\Z}{\mathbb{Z}}
\newcommand{\eps}{\varepsilon}

\makeatletter
\DeclareRobustCommand{\cev}[1]{%
	{\mathpalette\do@cev{#1}}%
}
\newcommand{\do@cev}[2]{%
	\vbox{\offinterlineskip
		\sbox\z@{$\m@th#1 x$}%
		\ialign{##\cr
			\hidewidth\reflectbox{$\m@th#1\vec{}\mkern4mu$}\hidewidth\cr
			\noalign{\kern-\ht\z@}
			$\m@th#1#2$\cr
		}%
	}%
}
\makeatother

\title{ \textbf{The number of ends in the uniform spanning tree for recurrent unimodular random graphs.}}

\author{Diederik van Engelenburg \and Tom Hutchcroft}

\begin{document}
	
\maketitle

\begin{abstract}
	We prove that if a unimodular random rooted graph is recurrent, the number of ends of its uniform spanning tree is almost surely equal to the number of ends of the graph. Together with previous results in the transient case, this completely resolves the problem of the number of ends of wired uniform spanning forest components in unimodular random rooted graphs and confirms a conjecture of Aldous and Lyons (2006).
\end{abstract}

\section{Introduction}
The \textbf{uniform spanning tree} of a finite connected graph $G$ is defined by picking uniformly at random a connected subgraph of $G$ containing all vertices but no cycles. 
To go from finite to infinite graphs, it is possible to exhaust $G$ by finite subgraphs and take weak limits with appropriate boundary conditions. For two natural such choices of boundary conditions, known as \textbf{free} and \textbf{wired} boundary conditions, Pemantle \cite{Pemantle} proved that these infinite-volume limits are always well-defined independently of the choice of exhaustion, and that the choice of boundary conditions also does not affect the limit obtained when $G=\Z^d$. Since connectivity of a subgraph is not a closed condition, these weak limits might be supported on configurations that are \emph{forests} rather than trees, and indeed Pemantle proved for $\Z^d$ that the limit is connected if and only if $d\leq 4$. For a general infinite, connected, locally finite graph the infinite-volume limit of the UST with free boundary conditions is called the \textbf{free uniform spanning forest} (FUSF) and the infinite volume limit with wired boundary conditions is called the \textbf{wired uniform spanning forest} (WUSF); when the two limits are the same we refer to them simply as the uniform spanning forest (USF). In their highly influential work~\cite{BLPS}, Benjamini, Lyons, Peres and Schramm  resolved the connectivity question for the WUSF in large generality: the wired uniform spanning tree is a single tree if and only if two random walks intersect infinitely often. The connectivity of the FUSF appears to be a much more subtle question and, outside of the case that the two forests are the same, is understood only in a few examples \cite{MR4010561,pete2022free,tang2021weights,AHNR}.
%
%
For \emph{recurrent} graphs, which are the main topic of this paper, the infinite-volume limit of the UST is always defined independently of boundary conditions and a.s.\ connected \cite[Proposition 5.6]{BLPS}, so that we can unambiguously refer to the uniform spanning tree (UST) of an infinite, connected, locally finite, recurrent graph $G$.


%


After connectivity, the next most basic topological property of the USF is the number of \textbf{ends} its components have.
Here, we say that a graph has at least $m$ ends whenever there exists some finite set of vertices $W$ such that $G \setminus W$ has at least $m$ infinite connected components. The graph is said to be $m$-ended if at has at least $m$ but not $m + 1$ ends. Understanding the number of ends of the USF turns out to be rather more difficult than connectivity, with a significant literature now devoted to the problem.
 For \emph{Cayley graphs}, it follows from abstract principles \cite[Section 3.4]{AHNR} that every component has $1$, $2$, or infinitely many ends almost surely, and for \emph{amenable} Cayley graphs such as $\Z^d$ (for which the WUSF and FUSF always coincide) is follows by a Burton-Keane \cite{BurtonKeane} type argument that every component has either one or two ends almost surely;
   see \cite[Chapter 10]{LyonsPeresProbNetworks} for detailed background.
   For the \emph{wired} uniform spanning forest on transitive graphs, a complete solution to the problem was given by Benjamini, Lyons, Peres, and Schramm \cite{BLPS} and Lyons, Morris, and Schramm \cite{LyonsMorrisSchramm2008}, who proved that every component of the WUSF of an infinite transitive graph is one-ended almost surely unless the graph in question is rough-isometric to $\Z$. Before going forward, let us emphasize that the recurrent case of this result \cite[Theorem 10.6]{BLPS} is established using a completely different argument to the transient case, with the tools available for handling the two cases being largely disjoint.


Beyond the transitive setting, various works have established mild conditions under which every component of the WUSF is one-ended almost surely, applying in particular to planar graphs with bounded face degrees \cite{MR4010561} and graphs satisfying isoperimetric conditions only very slightly stronger than transience \cite{MR3773383,LyonsMorrisSchramm2008}. These proofs are quantitative, and recent works studying critical exponents for the USF of $\Z^d$ with $d\geq 3$ \cite{MR4055195,hutchcroft2020logarithmic,MR4348685} and Galton-Watson trees~\cite{MR4095018} can be thought of as a direct continuation of the same line of research.





In parallel to this deterministic theory,
 Aldous and Lyons \cite{AldousLyonsUnimod2007} observed that the methods of \cite{BLPS} also apply to prove that the WUSF has one-ended components on any transient \emph{unimodular random rooted graph} of bounded degree, and the second author later gave new proofs of this result with different methods that removed the bounded degree assumption \cite{Hutchcroft2,MR3773383}. It is also proven in \cite{MR3651050,MR3813990} that every component of the \emph{free} uniform spanning forest of a unimodular random rooted graph is infinitely ended a.s.\ whenever the free and wired forests are different.
 Here, unimodular random rooted graphs comprise a very large class of random graph models including Benjamini-Schramm limits of finite graphs \cite{BenjaminiSchrammRecurrence}, Cayley graphs, and (suitable versions of) Galton-Watson trees, as well as e.g.\ percolation clusters on such graphs; See Section~\ref{subsec:definitions} for definitions and e.g.\ \cite{CurNotes,AldousLyonsUnimod2007} for detailed background. 


  The aforementioned works \cite{AldousLyonsUnimod2007,Hutchcroft2,MR3773383,MR3651050,MR3813990} completely resolved the problem of the number of ends of the WUSF and FUSF for \emph{transient} unimodular random rooted graphs, but the recurrent case remained open.
Besides the fact that the transient methods do not apply, a further complication of the recurrent case is that it is possible for the UST to be either one-ended or two-ended according to the geometry of the graph: indeed, the UST of $\Z^2$ is one-ended while the UST of $\Z$ is two-ended.

 Aldous and Lyons conjectured \cite[p.~1485]{AldousLyonsUnimod2007} that the dependence of the number of ends of the UST on the geometry of the graph is as simple as possible: The UST of a recurrent unimodular random rooted graph is one-ended if and only if the graph is. 
The fact that two-ended unimodular random rooted graphs have two-ended USTs is trivial; the content of the conjecture is that one-ended unimodular random rooted graphs have one-ended USTs. Previously, the conjecture was resolved under the assumption of planarity in \cite{AHNR}, while in \cite{BvE21} it was proved (without using the planarity assumption) that the UST of a recurrent unimodular random rooted graph is one-ended precisely when the ``harmonic measure from infinity’’ is uniquely defined. 
 In this paper we resolve the conjecture. 

\begin{theorem} \label{T:main}
	Let $(G, o)$ be a recurrent unimodular random rooted graph and let $T$ be the uniform spanning tree of $G$. Then $T$ has the same number of ends as $G$ a.s.
\end{theorem}

 To see that the theorem is not true without unimodularity, 
  consider taking the line graph $\ZZ$ and adding a path of length $2^n$ connecting $-n$ connecting to $n$ for each $n$, making the graph one-ended. Kirchoff's effective resistance formula implies that the probability that the additional path connecting $-n$ to $n$ is included in the UST is at most $n/(2^n+n)$, and a simple Borel-Cantelli argument implies that the UST is two-ended almost surely. Similar examples show that Theorem~\ref{T:main} does not apply to unimodular random rooted \emph{networks}, since we can use edges of very low conductance to make the network one-ended while having very little effect on the geometry of the UST. 


\medskip

\textbf{About the proof.}
We stress again that the tools used in the transient case do not apply at all to the recurrent case, and we are forced to use completely different methods that are specific to the recurrent case. 
We build on \cite{BvE21} which proved that the ``harmonic measures from infinity'' are uniquely defined if and only if the uniform spanning tree is one-ended; A self-contained treatment of (a slight generalization of) the results of \cite{BvE21} that we will need is given in Appendix~\ref{Appendix:BvE}. The set of harmonic measures from infinity can be thought of as a ``boundary at infinity'' for the graph, analogously to the way the Martin boundary is used in transient graphs. It is implicit in \cite{BvE21} that these measures correspond to the ways in which a random walk ``conditioned to never return to the root'' can escape to infinity. We develop these ideas further in Section~\ref{sec:boundary_theory}, in which we make this connection precise. We then apply these ideas inside an ergodic-theoretic framework to prove that \textit{if} the UST has two ends, then the effective resistance must grow linearly along the unique bi-infinite path in the tree, which implies in particular that graph distances must also grow linearly. To conclude, we argue that this can only happen when the graph has linear volume growth, which is known to be equivalent to two-endedness for unimodular random rooted graphs \cite{BenHutch,bowen2021perfect}.

\paragraph{Acknowledgments} DvE wishes to thank Nathana\"el Berestycki for many useful discussions and comments. Most of the research was carried out while DvE visited TH at Caltech and we are grateful for the hospitality of the institute. DvE is supported by the FWF grant P33083, “Scaling limits in random conformal geometry”.

\section{Boundary theory of recurrent graphs}
\label{sec:boundary_theory}

In this section we develop the theory of harmonic measures from infinity on recurrent graphs, their associated potential kernels and Doob transforms, and how this relates to the spanning tree. Much of the theory we develop here is a direct analogue for recurrent graphs of the theory of Martin boundaries of transient graphs \cite{Woess,dynkin1969boundary}.  This theory is interesting in its own right, and we were surprised to find how little attention has been paid to these notions outside of some key motivating examples such as $\Z^2$ \cite{popov2020transience,gantert2019range}.

\medskip

All of the results in this section will concern deterministic infinite, connected, recurrent, locally finite graphs; applications of the theory to unimodular random rooted graphs will be given in Section~\ref{sec:main_proof}.

\subsection{Harmonic measures from infinity}

Let $G=(V,E)$ be an infinite, connected, locally finite, recurrent graph. For each $v\in V$ we write $\mathbf{P}_v$ for the law of the simple random walk on $G$ started at $v$, and for each set $A\subseteq V$ write $T_A$ and $T^+_A$ for the first visit time of the random walk to $A$ and first positive visit time of the random walk to $A$ respectively. Given a probability measure $\mu$ on $V$, we also write $\mathbf{P}_\mu$ for the law of the random walk started at a $\mu$-distributed vertex. 

A \textbf{harmonic measure from infinity}  $h=(h_B: B\subset V$ finite$)$  on $G$ is a collection of probability measures on $V$ indexed by the finite subsets $B$ of $V$ with the following properties:
\begin{enumerate}
	\item $h_B$ is supported on $\partial B$ for each $B\subset V$, where $\partial B$ is the set of elements of $B$ that are adjacent to an element of $V\setminus B$.
	\item For each pair of finite sets $B \subseteq B'$, $h_B$ and $h_{B'}$ satisfy the consistency condition
	\begin{equation}
	\label{eq:consistency}
	h_B(u) = \sum_{v\in B'} h_{B'}(v)\mathbf{P}_v(X_{T_B}=u)
	\end{equation}
	for every $u\in B$.
\end{enumerate}
We denote the space of harmonic measures from infinity by $\c{H}$, which (identifying the measures $h_B$ with their probability mass functions) is a compact convex subset of the space of functions $\{$finite subsets of $V\}\to\mathbb{R}^V$ when equipped with the product topology. As mentioned above, the space $\c{H}$ plays a role for recurrent graphs analogous to that played by the Martin boundary for transient graphs; the analogy will become clearer once we introduce potential kernels in the next subsection. We say that the harmonic measure from infinity is \textbf{uniquely defined} when $\c{H}$ is a singleton.

If $\mu_n$ is a sequence of probability measures on $V$ converging vaguely to the zero measure in the sense that $\mu_n(v)\to 0$ as $n\to\infty$ for each fixed $v\in V$ then any subsequential limit of the collections $(\mathbf{P}_{\mu_n}(X_{T_B}  = \cdot )\,:\, B \subset V$ finite$)$ belongs to $\c{H}$, with these collections themselves satisfying every property of a harmonic measure from infinity other than the condition that $h_B$ is supported on $\partial B$ for every finite $B$. (Indeed, the consistency condition \eqref{eq:consistency} follows from the strong Markov property of the random walk.) In fact every harmonic measure from infinity can be written as such a limit.

\begin{lemma}
\label{lem:harmonic_measures_as_limits}
If $h\in \c{H}$ is a harmonic measure from infinity then there exists a sequence of finitely supported probability measures $(\mu_n)_{n\geq 1}$ on $V$ such that $\mu_n(v)\to 0$ for every $v\in V$ and 
\begin{equation} \label{eq: HM def}
	 h_B(\cdot) = \lim_{n \to \infty} \mathbf{P}_{\mu_n}(X_{T_B}  = \cdot\, ) \qquad \text{ for every $B\subset V$ finite.}
\end{equation}
\end{lemma}

\begin{proof}
Fix $h\in \c{H}$.
Let $V_1 \subset V_2 \subset V_3 \cdots$ be an increasing sequence of finite subsets of $V$ with $\bigcup_i V_i = V$, and for each $n\geq 1$ let $\mu_n=h_{V_n}$. It follows from the consistency condition \eqref{eq:consistency} that
\[
h_B(\cdot) = \mathbf{P}_{\mu_n}(X_{T_B}  = \cdot\, ) \qquad \text{ for every $B\subset V_n$,}
\]
and the claim follows since every finite set is eventually contained in $V_n$.
\end{proof}

Since $\c{H}$ is a weakly compact subspace of the set of functions from finite subsets of $V$ to $\mathbb{R}^V$, which is a locally convex topological vector space, it is a Choquet-simplex: Every element can be written as a convex combination of the extremal points. In particular, if $\c{H}$ has more than one point then it must have more than one extremal point. 
 This will be useful to us because extremal points of $\c{H}$ are always limits of harmonic measures from sequences of single vertices. Indeed, identifying each vertex $v\in V$ with the collection of harmonic measures $(\mathbf{P}_v(X_{T_B}=\cdot):B\subset V$ finite$)$ allows us to think of $V \cup \c{H}$ as a compact Polish space containing $V$ (in which $V$ might not be dense), and we say that a sequence of \emph{vertices} $(v_n)_{n\geq 0}$ converges to a point $h\in \c{H}$ if $h_B(\cdot) = \lim_{n \to \infty} \mathbf{P}_{v_n}(X_{T_B}  = \cdot \,)$ for every $B\subset V$ finite.

\begin{lemma} \label{L: extremal generators}
	If $h \in \c{H}$ is extremal, there exists a sequence of vertices $(v_n)_{n\geq 0}$ such that $v_n$ converges to $h$ as $n\to\infty$.
\end{lemma}

\begin{proof}
Let $\c{I}$ be the set of functions $h: \{B \subset V \text{ finite}\}\to \RR^V$ of the form
	 \[h_B(\cdot) =  \mathbf{P}_{\mu}(X_{T_B}  = \cdot ) \qquad \text{ for every $B\subset V$ finite}\]
	 for some finitely supported measure $\mu$ on $V$. Lemma~\ref{lem:harmonic_measures_as_limits} implies that $\overline{\c{I}}=\c{I}\cup\c{H}$ is a compact convex subset of the  space of all functions $\{B \subset V \text{ finite}\}\to \mathbb{R}^V$ equipped with the product topology, which is a locally convex topological vector space. By the Krein-Milman theorem, a subset $W$ of $\c{I} \cup \c{H}$ has closure containing the set of extremal points of $\c{I} \cup \c{H}$ if and only if $\c{I} \cup \c{H}$ is contained in the closed convex hull of $W$. Thus, if we define $\mathcal{I}_\mathrm{ext}$ to be the set of functions
	 $h: \{B \subset V \text{ finite}\}\to \RR^V$ of the form
	 \[h_B(\cdot) =  \mathbf{P}_{z}(X_{T_B}  = \cdot ) \qquad \text{ for every $B\subset V$ finite}\]
	 for some $z\in V$ then $\c{I}$ is clearly contained in the convex hull of $\mathcal{I}_\mathrm{ext}$, so that $\c{I}\cup\c{H}$ is contained in the closed convex hull of $\mathcal{I}_\mathrm{ext}$ and, by the Krein-Milman theorem, the set of extremal points of $\c{I}\cup\c{H}$ is contained in the closure of $\mathcal{I}_\mathrm{ext}$. 

	 Now, observe that for any non-trivial convex combination of an element of $\c{I}$ and an element of $\c{H}$, there must exist a finite set of vertices $B$ and a point $z$ in the interior of $B$ (i.e., in $B$ and not adjacent to any element of $V\setminus B$) such that $h_B(z)\neq 0$; indeed, if the element of $\c{I}$ corresponds to some finitely supported measure $\mu$, then any $B$ containing the support of $\mu$ in its interior and any $z$ in the support of $\mu$ will do. Since no element of $\c{H}$ can have this property, it follows that non-trivial convex combinations of elements of $\c{I}$ and $\c{H}$ cannot belong to $\c{H}$ and hence that extremal points of $\c{H}$ are also extremal in $\c{I}\cup\c{H}$. It follows that the set of extremal points of $\c{H}$ is contained in the closure of $\c{I}_\mathrm{ext}$, which is equivalent to the claim.
\end{proof}

	\begin{remark} The converse to this lemma is \emph{not} true: A limit of a sequence of Dirac measures need not be extremal. For example, if we construct a graph from $\Z$ by attaching a very long path between $-n$ and $n$ for each $n\geq 1$ and take $z_n$ to be a point in the middle of this path for each $n$, the sequence $(z_n)_{n\geq 1}$ will converge to a non-extremal element of $\c{H}$ that is the convex combination of the limits of $(n)_{n\geq 1}$ and $(-n)_{n\geq 1}$.
	\end{remark}

\subsection{Potential kernels and Doob transforms}


The arguments in \cite{BvE21} heavily rely on a correspondence between the harmonic measure from infinity and its \textbf{potential kernel}. One important feature of the potential kernel is that, given a vertex $o\in V$ and a point $h\in\c{H}$, it provides a sensible way to ``condition the random walk to converge to $h$ before returning to $o$". We begin by discussing how conditioning the random walk to hit a particular vertex before returning to $o$ can be described in terms of Doob transforms before developing the analogous limit theory.

\medskip

\textbf{Doob transforms and non-singular conditioning.}
Suppose that we are given two distinct vertices $o$ and $z$ in an infinite, connected, locally finite recurrent graph $G$. Letting $\mathbf{G}_z(x,y)$ be the expected number of times a random walk started at $x$ visits $y$ before hitting $z$, we can compute that the function
\[
a(x)=\frac{\mathbf{G}_z(o,o)}{\deg(o)}-\frac{\mathbf{G}_z(x,o)}{\deg(o)} = \frac{\mathbf{P}_x(T_o>T_z)\mathbf{G}_z(o,o)}{\deg(o)}
\]
is harmonic at every vertex other than $o$ and $z$, and has
\[
\Delta a(o) = 0-\deg(o) \mathbf{E}_o[a(X_1)] = -\mathbf{P}_o(T_o^+>T_z)\mathbf{G}_z(o,o) = -1,
\]
where $\Delta$ denotes the graph Laplacian $\Delta f(x)=\deg(x)f(x)-\sum_{y\sim x} f(y)=\deg(x)\mathbf{E}_x[f(X_0)-f(X_1)]$ (terms in this sum are counted with appropriate multiplicity if there is more than one edge between $x$ and $y$). 
Moreover, the quantity $a(x)$ is strictly positive at every vertex $x$ that is neither equal to $o$ nor disconnected from $z$ by $o$ in the sense that every path from $x$ to $z$ must pass through $o$. 
Observe that the trivial identity
\begin{multline}
\mathbf{P}_o((X_0,\ldots,X_{n})=(x_0,\ldots,x_n) ) = \prod_{i=1}^n p(x_{i-1},x_i) \\= \frac{1}{a(x_n)} a(x_1)p(o,x_1) \prod_{i=2}^n \frac{a(x_i)}{a(x_{i-1})} p(x_{i-1},x_i)\label{eq:finite_Doob0}
\end{multline}
holds 
for every sequence of vertices $x_0,\ldots,x_n$ with $x_0=o$ and $a(x_i)>0$ for every $i>0$. Since $a(z)=G_z(o,o)=\mathbf{P}_o(T_z<T_o^+)^{-1}$ it follows that
\begin{equation}
\label{eq:finite_Doob}
\mathbf{P}_o((X_0,\ldots,X_{n})=(x_0,\ldots,x_n) \mid T_z<T_o^+) = a(x_1)p(o,x_1) \prod_{i=2}^n \frac{a(x_i)}{a(x_{i-1})} p(x_{i-1},x_i)
\end{equation}
for every sequence of vertices $x_0,\ldots,x_n$ with $x_0=o$, $x_n=z$, and $x_i \notin \{o,z\}$ for every $0<i<n$ (which implies that $a(x_i)>0$ for every $1\leq i \leq n$). Now, the fact that $a$ is harmonic off of $\{o,z\}$ and has $\Delta a(o)=-1$ implies that we can define a stochastic matrix with state space $\{x\in V: x=o$ or $a(x)>0\}$ by
\[
\wh{p}^a(x,y)=\begin{cases}
\frac{a(y)}{a(x)}p(x,y) & x \notin \{o,z\}\\
a(y) & x =0\\
\id(y=z) & x = z,
\end{cases}
\]
and if we define the \textbf{Doob transformed walk} $\wh{X}^a$ to be the Markov chain with this transition matrix started from $o$ then it follows from \eqref{eq:finite_Doob} that $(\wh{X}^a_{n})_{n=0}^{T_z}$ has law equal to the conditional law of the simple random walk $(X_n)_{n=0}^{T_z}$ started at $o$ and conditioned to hit $z$ before returning to $o$.
Moreover, letting $\wh{\mathbf{P}}^a_o$ denote the law of $\wh{X}^a$, it follows from the definition of $\wh{X}^a$ that
\begin{multline}
\wh{\mathbf{P}}^a_o\left ((\wh{X}^a_0,\ldots,\wh{X}^a_n)=(x_0,\ldots,x_n)\right) = 
\prod_{i=1}^n \wh{p}^a(x_{i-1},x_i)= 
a(x_{1})\prod_{i=2}^n \frac{a(x_i)}{a(x_{i-1})} p(x_{i-1}, x_i) \\
= a(x_{n})\prod_{i=2}^n p(x_{i-1}, x_i)
=
\deg(o) a(x_n) \mathbf{P}_o\left((X_0,\ldots,X_n)=(x_0,\ldots,x_n)\right)
\label{eq:finite_Doob2}
\end{multline}
for every sequence $x_0,\ldots,x_n$ with $x_0=o$ and $x_i\notin\{o,z\}$ for every $0<i<n$. 

\paragraph{Defining the potential kernel.}
We now define the \textbf{potential kernel} $a^h$ associated to a point $h\in \c{H}$ via the formula
\begin{equation} \label{eq: PK def}
	a^h(x, y) =  h_{x, y}(x) \Reff(x \leftrightarrow y)  
\end{equation}
where we write $h_{x,y}=h_{\{x,y\}}$, so that $a^h(x,x)=0$ for each $x\in V$. The fact that this is a sensible definition owes largely to the following lemma.

\begin{lemma}
\label{lem:PK_harmonic}
For each $h\in \c{H}$, the potential kernel $a^h(x, y) =  h_{x, y}(x) \Reff(x \leftrightarrow y)$ satisfies
\begin{equation} \label{eq: PK harmonic}
	\Delta a^h(\,\cdot\,, y) = -\id(\,\cdot=y),
\end{equation}
so that the potential kernel $a^h(\cdot, y)$ is harmonic away from $y$ and subharmonic at $y$.
\end{lemma}

\begin{proof}
Since the map $h\mapsto a^h$ is affine and the equality \eqref{eq: PK harmonic} is linear, it suffices to prove the lemma in the case that $h$ is extremal. 
By Lemma~\ref{L: extremal generators}, there exists a sequence of vertices $(v_n)_{n\geq 1}$ such that $v_n$ converges to $h$. 
For each $n \geq 1$ we define
\[
a^n(x,y)=\frac{\Gr_{v_n}(y, y)}{\deg(y)} - \frac{\Gr_{v_n}(x, y)}{\deg(y)}.
\]
 and claim that
\begin{equation}
\label{eq:PK_limit}
	a^h(x, y) = \lim_{n \to \infty} a^n(x,y)
\end{equation}
for every $x,y\in V$.
(Note that this limit formula is often taken as the \emph{definition} of the potential kernel.) 
We will prove \eqref{eq:PK_limit} with the aid of three standard identities for the Greens function:
\begin{enumerate}
	\item By the strong Markov property, $\Gr_{z}(x, y)$ is equal to $\mathbf{P}_x(T_y<T_z)\Gr_{z}(y, y)$ for every three distinct vertices $x$, $y$, and $z$.
	\item By the strong Markov property, $\Gr_x(y,y)$ is equal to $\mathbf{P}_x(T_y<T_x^+)^{-1}$ for every pair of distinct vertices $x$ and $y$. It follows in particular that $\deg(y)^{-1}\Gr_x(y,y)=\Reff(x\leftrightarrow y)$ and,  since the effective resistance is symmetric in $x$ and $y$, that $\deg(y)^{-1}\Gr_x(y,y) = \deg(x)^{-1}\Gr_y(x,x)$.
	\item By time-reversal, $\deg(x)\Gr_{z}(x, y)$ is equal to $\deg(y)\Gr_{z}(y, x)$ for every three distinct vertices $x$, $y$, and $z$.
\end{enumerate}
Applying these three identities in order yields that
%
\begin{align*}
a^n(x,y)&=\frac{\Gr_{v_n}(y, y)}{\deg(y)}\mathbf{P}_x(T_y> T_{v_n})
=\frac{\Gr_{y}(v_n, v_n)}{\deg(v_n)}\mathbf{P}_x(T_y> T_{v_n}) &= \frac{\Gr_{y}(x, v_n)}{\deg(v_n)} = \frac{\Gr_{y}(v_n, x)}{\deg(x)}
\end{align*}
whenever $x$, $y$, and $v_n$ are distinct.
Applying the first and second identities a second time then yields that
\begin{align}
a^n(x,y)
= \mathbf{P}_{v_n}(T_x<T_y)\frac{\Gr_{y}(x, x)}{\deg(x)} = \mathbf{P}_{v_n}(T_x<T_y) \Reff(x\leftrightarrow y)
\end{align}
whenever $x$, $y$, and $v_n$ are distinct. This is easily seen to imply the claimed limit formula \eqref{eq:PK_limit}.

\end{proof}

In light of this lemma, we define $\c{P}_o$ to be the space of non-negative functions $a:V\to[0,\infty)$ with $a(o)=0$ and $\Delta a(x) =-\id(x=o)$, so that $a^h(\,\cdot\,,o)$ belongs to $\c{P}_o$ for each $o\in V$ and $h\in \c{H}$ by Lemma~\ref{lem:PK_harmonic}. We will later show that the map $h\mapsto a^h$ is an affine isopmorphism between the two convex spaces $\c{H}$ and $\c{P}_o$. We first describe how elements of $\c{P}_o$ can be used to define Doob transformed walks.

\paragraph{Doob transforms and singular conditioning.} We now define the Doob transform associated to an element of the space $\c{P}_o$.
Given $a\in \c{P}_o$, we define $\wh{X}^a$ to be the Doob $a$-transform of the simple random walk $X$ on $G$, so that $\wh{X}^a$ has state space $\{x\in V: x=o$ or $a(x)>0\}$ and transition probabilities given by
\[
	\wh{p}^{\: a}(x, y) := \begin{cases}
		\frac{a(y)}{a(x)} p(x, y) & \text{if } x \neq o \\
		a(y) & \text{if } x = o, y \sim o
	\end{cases}
\]
where $p$ is the transition kernel for the simple random walk. Similarly, given $h \in \c{H}$, we write $\wh{X}^h=\wh{X}^{a^h(\cdot,o)}$ where $a^h$ is the potential kernel associated to $h$.
Informally, we think of $\wh{X}^h$  as the walk that is ``conditioned to go to $h$ before returning to $o$''. (In particular, when the harmonic measure from infinity is unique and $\c{H}$ and $\c{P}_o$ are singleton sets, we think of the associated Doob transform as the random walk conditioned to never return to $o$.) We write $\wh{\mathbf{P}}_o^a$ or $\wh{\mathbf{P}}^h_o$ for the law of $\wh{X}^a$ or $\wh{X}^h$ as appropriate.


As before, it follows from this definition that if $a\in \c{P}_o$ and we write $X[0,m]$ for the initial segment consisting of the first $m$ steps of the random walk $X$ then
\begin{multline}
\label{eq:Doob_transform_whole_path}
\wh{\mathbf{P}}^a_o(\wh{X}^a[0,m]=\gamma)=\prod_{i=1}^m \wh{p}^{\: a}(\gamma_{i-1}, \gamma_i) =
a(\gamma_{1})\prod_{i=2}^m \frac{a(\gamma_i)}{a(\gamma_{i-1})} p(\gamma_{i-1}, \gamma_i) \\
= a(\gamma_{m})\prod_{i=2}^m p(\gamma_{i-1}, \gamma_i)
=
 \deg(o)a(\gamma_m) \mathbf{P}_o(X[0, m] = \gamma)
\end{multline}
for every finite path $\gamma=(\gamma_0,\ldots,\gamma_m)$ with $\gamma_0=o$ and $\gamma_i \neq o$ for every $i>0$. Summing over all paths $\gamma$ that begin at $o$, end at some point $x\neq o$, and do not visit $o$ or $x$ at any intermediate point yields in particular that if $h\in \c{H}$ then
\begin{equation}
\label{eq:hitting prop CRW}
\wh{\mathbf{P}}^h_o(\wh{X} \text{ hits $x$}) = \deg(o) a^h(x,o) \mathbf{P}_o(T_x < T^+_o) = h_{o,x}(x),
\end{equation}
where the last equality follows from \eqref{eq: PK def} and the definition of the effective resistance.

\begin{lemma}
Let $G=(V,E)$ be a recurrent graph and let $a\in \c{P}_o$. Then the associated Doob-transformed walk $\wh{X}^a$ is transient.
\end{lemma}

\begin{proof}
One can easily verify from the definitions that the sequence of reciprocals
$(a(\wh{X}^a_n)^{-1})_{n\geq 1}$
is a non-negative martingale with respect to its natural filtration, and hence converges almost surely to some limiting random variable, which it suffices to prove is zero almost surely.
It follows from the identity \eqref{eq:Doob_transform_whole_path} that
\[
\wh{\mathbf{P}}^a_o(a(\wh{X}^a_n)\leq M) = \sum_{v} \id(a(v)\leq M) \deg(o)a(v) \mathbf{P}_o(X_n=v, T^+_o > n) \leq M\deg(o)\mathbf{P}_o(T^+_o > n),
\]
for every $n,M\geq 1$. Since $G$ is recurrent, the right hand side tends to zero as $n\to\infty$ for each fixed $M$. It follows that $\limsup_{n\to\infty} a(\wh{X}^a_n) = \infty$ almost surely, and hence that $\lim_{n\to\infty} a(\wh{X}^a_n)=\infty$ almost surely since the limit is well-defined almost surely. This implies that $\wh{X}^a$ is transient.
\end{proof}



\subsection{An affine isomorphism}

Let $G=(V,E)$ be recurrent, fix $o\in V$, and let $\c{P}_o$ denote the set of positive functions $a:V\to [0,\infty)$ with $a(o)=0$ that satisfy $\Delta a(\cdot) = -\id(\,\cdot=o)$. As we have seen, for each $h\in \c{H}$ the potential kernel $a^h(\cdot,o)$ defines an element of $\c{P}_o$. Moreover, the map sending $h\mapsto a^h(\,\cdot\,,o)$ is affine in the sense that if $h=\theta h_1 + (1-\theta) h_2$ then $a^h(\,\cdot\,,o)=\theta a^{h_1}(\,\cdot\,,o)+(1-\theta)a^{h_2}(\,\cdot\,,o)$. We wish to show that this map defines an affine \emph{isomorphism} between $\c{H}$ and $\c{P}_o$ in the sense that it is bijective (in which case its inverse is automatically affine). We begin by constructing the inverse map from $\c{P}_o$ to $\c{H}$.


\begin{lemma}
\label{lem:PK_surjectivity}
Let $G=(V,E)$ be a infinite, connected, locally finite recurrent graph and let $o\in V$. For each $a\in \c{P}_o$ 
there exists a unique $h\in \c{H}$ satisfying
\[h_B(u)=\wh{\mathbf{P}}^a_o(\wh{X}^a \text{ visits $B$ for the last time at $u$})\]
for every finite set $B$ containing $o$. Moreover, this $h$ satisfies $a^h(x,o)=a(x)$ for every $x\in V$.
\end{lemma}

\begin{proof}[Proof of Lemma~\ref{lem:PK_surjectivity}]
Fix $a\in \c{P}_o$. We define a the family of probability measures $h=(h_B:B\subset V$ finite$)$ by
\[h_B(u)=\wh{\mathbf{P}}^a_o(\wh{X}^a \text{ visits $B$ for the last time at $u$})\]
for every $u\in B$ if $o \in B$ and
\begin{multline*}h_B(u)=\wh{\mathbf{P}}^a_o(\wh{X}^a \text{ visits $B\cup\{o\}$ for the last time at $u$})\\+
\wh{\mathbf{P}}^a_o(\wh{X}^a \text{ visits $B\cup\{o\}$ for the last time at $o$})\mathbf{P}_o(X_{T_B}=u)
\end{multline*}
for every $u\in B$ if $o\notin B$, so that if $o\notin B$ then
\[
h_B(u)=\sum_{v\in B\cup\{o\}}h_{B\cup\{o\}}(v)\mathbf{P}_v(X_{T_B}=u)
\]
for every $u\in V$.
We claim that this defines an element of $\c{H}$. It is clear that $h_B$ is a probability measure that is supported on $\partial B$ for each finite set $B\subset V$; we need to verify that it satisfies the consistency property \eqref{eq:consistency}. Once it is verified that $h\in \c{H}$, the fact that $a=a^h(\cdot,o)$ follows easily from the definition of $a^h$ together with the identity \eqref{eq:Doob_transform_whole_path}, which together yield that
\begin{align*}
a^h(v,o)&=h_{v,o}(v)\Reff(v \leftrightarrow o) = \frac{\wh{\mathbf{P}}^a_o(\wh{X}^a \text{ visits $\{o,v\}$ for the last time at $v$}) }{\deg(o)\mathbf{P}_o(T_v<T^+_o)}\\
&=\frac{\wh{\mathbf{P}}^a_o(\wh{X}^a \text{ hits $v$}) }{\deg(o)\mathbf{P}_o(T_v<T^+_o)}=\frac{\deg(o)a(v)\mathbf{P}_o(T_v<T^+_o)}{\deg(o)\mathbf{P}_o(T_v<T^+_o)}=a(v)
\end{align*}
for each $v\in V$.

We now prove that $h$ satisfies the consistency property \eqref{eq:consistency}. We will prove the required identity in the case $o\in B$, the remaining case $o\notin B$ following from this case and the definition.
Let $B\subseteq B'$ be finite sets with $o\in B$ and let $(V_n)_{n\geq 1}$ be an exhaustion of $V$ by finite sets such that $B' \subseteq V_n$ for every $n\geq 1$. Writing $V_n^c=V\setminus V_n$ for each $n\geq 1$ and $\tau_n$ for the first time the walk visits $V_n^c$, we have that
\begin{align*}
h_B(u) &= \lim_{n\to\infty} \wh{\mathbf{P}}^a_o(\wh{X}[0,\tau_n] \text{ last visits $B$ at $u$})\\
&= \lim_{n\to\infty}\sum_{b \in V_n^c} \wh{\mathbf{P}}^a_o(\wh{X}[0,\tau_n] \text{ last visits $B$ at $u$, $\wh{X}_{\tau_n}=b$})
\end{align*}
and hence by \eqref{eq:Doob_transform_whole_path} and time-reversal that
\begin{align}
h_B(u)
&= \lim_{n\to\infty}\sum_{b \in V_n^c} \deg(o)a(b) \mathbf{P}_o(X[0,\tau_n] \text{ last visits $B$ at $u$, $X_{\tau_n}=b$})
\nonumber\\
&= \lim_{n\to\infty}\sum_{b \in V_n^c} \deg(b)a(b) \mathbf{P}_b(X_{T_B}=u, T_o<T_{V_n^c}^+).\label{eq:h_B_time_reverse}
\end{align}
It follows from this together with the strong Markov property that
\begin{align*}
h_B(u) &= \lim_{n\to\infty}\sum_{v\in B'} \sum_{b \in V_n^c} \deg(b)a(b) \mathbf{P}_b(X_{T_B'}=v, X_{T_B}=u, T_o<T_{V_n^c}^+) \\
&=\lim_{n\to\infty}\sum_{v\in B'} \sum_{b \in V_n^c} \deg(b)a(b) \mathbf{P}_b(X_{T_B'}=v, T_{B'}<T_{V_n^c}^+)\mathbf{P}_v(X_{T_B}=u, T_{o}<T_{V_n^c}^+).
\end{align*}
Now, we have by the strong Markov property that for each $b \in V_n^c$ and $v \in B'$
\begin{align*}
	\bf{P}_b(X_{T_{B'}} = v, T_{o} < T_{V_n^c}^+) = \bf{P}_b(X_{T_{B'}} = v, T_{B'} < T_{V_n^c}^+)\bf{P}_v(T_{V_n^c} > T_o).
\end{align*}
and by recurrence that $\lim_{n\to\infty}\mathbf{P}_v(T_{o}<T_{V_n^c}^+)=1$, so that
\begin{align*} \label{eq:h_B'_time_reverse}
h_B(u) =\lim_{n\to\infty}\sum_{v\in B'} \sum_{b \in V_n^c} \deg(b)a(b) \mathbf{P}_b(X_{T_B'}=v, T_{o}<T_{V_n^c}^+)\mathbf{P}_v(X_{T_B} = u).
\end{align*}
The claimed identity \eqref{eq:consistency} follows from this together with the identity \eqref{eq:h_B_time_reverse} applied to the larger set $B'$.\qedhere

\end{proof}

\begin{theorem}
\label{cor:affine_isomorphism}
Let $G$ be an infinite, recurrent, locally finite graph, and let $o\in V$. The map $h\mapsto a^h(\cdot,o)$ is an affine isomorphism $\c{H}\to \c{P}_o$. In particular, this map identifies extremal elements of $\c{H}$ with extremal elements of $\c{P}_o$.
\end{theorem}

\begin{proof}
It remains only to prove that $h\mapsto a^h$ is injective. To prove this it suffices by definition of $a^h$ to prove that $h_B$ is determined by $(h_{x,o}(x):x\in \partial B)$ for each finite set $B\subset V$ containing the vertex $o$. Fix one such set $B$. We have by definition of $\c{H}$ that
\[h_{x,o}(x)=\sum_{y\in \partial B} h_B(y) \mathbf{P}_y(T_x<T_o)=\sum_{y\in \partial B} A(x,y) h_B(y)\]
for each $x\in \partial B$ where $A(x,y):=\mathbf{P}_y(T_x<T_o)$ for each $x,y\in \partial B$, so that it suffices to prove that the matrix $A$ (which is indexed by $\partial B$) is invertible. Define a matrix $Q$ indexed by $\partial B$ by
\[
Q(x,y)=\mathbf{P}_y(T_{\partial B}^+<T_o, X_{T_{\partial B}^+}=x).
\]
Then we have by the strong Markov property that
\begin{multline*}
A(x,y) - \id(x=y)\mathbf{P}_x(T_x^+\geq T_o) = \mathbf{P}_y(T_x^+<T_o) \\= \sum_{z\in \partial B} \mathbf{P}_z(T_x<T_o)Q(z,y) = \mathbf{P}_x(T_x^+\geq T_o)Q(x,y) +\sum_{z\in \partial B} \mathbf{P}_z(T_x^+<T_o)Q(z,y)
\end{multline*}
and hence inductively that
\[
\mathbf{P}_y(T_x^+<T_o) = \mathbf{P}_x(T_x^+\geq T_o)\sum_{i=1}^n Q^n(x,y)+\sum_{z\in \partial B} \mathbf{P}_z(T_x^+<T_o)Q^n(z,y)
\]
for every $n\geq 1$. Since $Q$ is irreducible and substochastic, we can take the limit as $n\to\infty$ to obtain that
\[
A(x,y) = \id(x=y)\mathbf{P}_x(T_x^+\geq T_o) + \mathbf{P}_y(T_x^+<T_o) = \mathbf{P}_x(T_x^+\geq T_o) \sum_{i=0}^\infty Q^n(x,y)
\]
for every $x,y\in \partial B$.
It follows by a standard argument that the matrix $A$ is invertible with inverse $A^{-1}=\mathbf{P}_x(T_x^+\geq T_o)^{-1}(1-Q)$ as required.
\end{proof}


\subsection{The Liouville property for extremal Doob transforms}

In this section we prove a kind of tail-triviality property of the Doob-transformed walk corresponding to an extremal point $h\in \c{H}$. Letting $G=(V,E)$ be a graph, we recall that an event $A\subseteq V^\NN$ is said to be \emph{invariant} if $(x_0,x_1,\ldots)\in A$ implies that $(x_1,x_2,\ldots)\in A$ for every $(x_0,x_1,\ldots)\in V^\NN$. 

\begin{theorem}
\label{thm:Liouville}
Let $G=(V,E)$ be an infinite, connected, recurrent, locally finite graph and let $o\in V$.
If $h\in \c{H}$ is extremal then the Doob transformed random walk $\hat{X}^h$ does not have any non-trivial invariant events: If $A \subseteq V^\NN$ is an invariant event then $\wh{\mathbf{P}}^h_o(A)\in \{0,1\}$.
\end{theorem}

\begin{proof}
It suffices to prove the corresponding statement for $\wh{X}^a$ when $a$ is an extremal element of $\c{P}_o$.
Suppose not, so that $A$ is a non-trivial invariant event. We have by Levy's 0-1 law that 
\begin{equation}
\label{eq:Levy}
\mathbf{P}_o(\wh{X}^a\in A  \mid \wh{X}^a_1,\ldots,\wh{X}^a_n)\to \id(\wh{X}^a\in A) \text{ almost surely as $n\to\infty$}.
\end{equation} 
Moreover, we also have by invariance that
\[
\wh{\mathbf{P}}^a_x(\wh{X}^a \in A) = \sum_{y \in V} \frac{a(y)}{a(x)} p(x,y) \wh{\mathbf{P}}^a_y(\wh{X}^a \in A)
\]
and that
\[
\wh{\mathbf{P}}^a_o(\wh{X}^a\in A) = \sum_{y \in V} a(y) \wh{\mathbf{P}}^a_y(\wh{X}^a \in A).
\]
Since similar inequalities hold when we replace $A$ by $A^c$ it follows that we can write $a$ as a non-trivial convex combination of two elements of $\c{P}_o$
\[
a(x)= \wh{\mathbf{P}}^a_o(\wh{X}^a \in A)  \cdot \frac{a(x)\wh{\mathbf{P}}^a_x(\wh{X}^a \in A)}{\wh{\mathbf{P}}^a_o(\wh{X}^a \in A) } + \wh{\mathbf{P}}^a_o(\wh{X}^a \notin A)  \cdot \frac{a(x)\wh{\mathbf{P}}^a_x(\wh{X}^a \notin A)}{\wh{\mathbf{P}}^a_o(\wh{X}^a \notin A) },
\]
these two factors being different by \eqref{eq:Levy}, contradicting extremality of $a$.
\end{proof}

\begin{remark}
Underlying this proposition is the fact that once we fix $a\in\c{P}_o$, we can identify $\c{P}_o$ with the Martin boundary of the conditioned walk $\wh{X}^a$. Theorem~\ref{thm:Liouville} is the recurrent version of the fact that Doob transforming by an extremal element of the Martin boundary yields a process with trivial invariant sigma-algebra.
\end{remark}

For our purposes, the most important output of the Liouville property is the following proposition, which lets us easily tell apart the trajectories of two different Doob transformed walks $\wh{X}^h$ and $\wh{X}^{h'}$ by looking at any infinite subset of their traces (and, in particular, from their loop-erasures).

\begin{proposition}
\label{prop:identifying_h_from_trace}
Let $h,h'$ be distinct extremal elements of $\c{H}$ and let $\wh{X}^h$ be the Doob-transformed simple random walk corresponding to $h$. Then 
\[
\frac{a^{h'}(\wh{X}^h_n,o)}{a^h(\wh{X}^h_n,o)} \to 0 
\]
almost surely as $n\to\infty$. 
\end{proposition}


\begin{proof}
We prove the corresponding statement in which $a,a'$ are distinct extremal elements of $\c{P}_o$. Let $\wh{X}$ and $\wh{X}'$ have laws $\wh{\mathbf{P}}_o^a$ and $\wh{\mathbf{P}}_o^{a'}$ respectively.
One can easily verify from the definitions that
\[
(Z_n)_{n\geq 1}=\left(\frac{a'(\wh{X}_n)}{a(\wh{X}_n)}\right)_{n\geq 1} \qquad \text{ and } \qquad (Z_n')_{n\geq 1}=\left(\frac{a(\wh{X}'_n)}{a'(\wh{X}'_n)}\right)_{n\geq 1}
\]
are both non-negative martingales with respect to their natural filtrations, and hence converge almost surely to some limiting random variables $Z$ and $Z'$. Since $Z$ and $Z'$ are measurable with respect to the invariant $\sigma$-algebras of $\wh{X}$ and $\wh{X}'$ respectively and $a$ and $a'$ are both extremal, there must exist constants $\alpha$ and $\alpha'$ such that $Z=\alpha$ and $Z'=\alpha'$ almost surely.
  We also have that $\E Z_n = \E Z_n ' = 1$ for every $n\geq 1$ and hence that $\alpha,\alpha'\leq 1$. We wish to prove that $\alpha=0$.

It follows from \eqref{eq:Doob_transform_whole_path} that the conditional distributions of the initial segments $\wh{X}[0,m]$ and $\wh{X}'[0,m]$ are the same if we condition on $\wh{X}_m=\wh{X}'_m=v$ for any $v\in V$ for any $v\in V$ and $m\geq 1$ and that 
\[
\frac{\mathbf{P}_o^a(\wh{X}_m=v)}{\mathbf{P}_o^{a'}(\wh{X}^{a'}_m=v)} = \frac{a(v)}{a'(v)}
\]
for every $m\geq 1$ and $v\in V$. If $\alpha>0$ then for every $\varepsilon>0$ there exists $M$ such that the distribution of $\wh{X}_m$ puts mass at least $1-\varepsilon$ on the set of vertices with $a'(v)/a(v) \geq (1-\varepsilon)\alpha $ for every $m\geq M$, and it follows that for each $m\geq M$ there is a coupling of the two walks $\wh{X}'$ and $\wh{X}$ so that their initial segments of length $m$ coincide with probability at least $(1-\varepsilon)^2\alpha$. Taking a weak limit as $m\to\infty$ and $\varepsilon \to 0$, it follows that there exists a coupling of the two walks $\wh{X}'$ and $\wh{X}$ such that the two walks coincide forever with probabilty at least $\alpha>0$. If we couple the walks in this way then on this event we must have that $Z'=1/Z$, which can occur with positive probability only if $\alpha'=1/\alpha$. Since $\alpha,\alpha'\leq 1$ we must have that $\alpha=\alpha'=1$ and that we can couple the two walks to be exactly the same almost surely. This is clearly only possible if $a=a'$, and since  $a\neq a'$ by assumption we must have that $\alpha=0$.
\end{proof}




\subsection{Potential kernels and the uniform spanning tree}

 We now use Lemma~\ref{L: local convergence} to show that the UST of a recurrent graph can always be sampled using a variant of Wilson's algorithm \cite{WilsonAlgorithm,BLPS} in which we `root at a point in $\c{H}$', where again we are thinking intuitively of $\c{H}$ as a kind of boundary at infinity of the graph.
Fix $h \in \mathrm{ex}(\c{H})$ and let $\wh{X}^h$ be the conditioned walk of the previous section. Fix some enumeration $V=\{v_1,v_2,\ldots\}$ of $V$ with  $v_1 = o$. Set $E_0 = \LE(\wh{X}^h[0, \infty))$ (which is well defined because $\wh{X}^h$ is transient) and for each $i\geq 1$ define $E_i$ given $E_{i- 1}$ recursively as follows:
\begin{itemize}
	\item if $v_i \in E_{i - 1}$, set $E_i = E_{i - 1}$ 
	\item otherwise, set $E_i = E_{i -1} \cup \LE(Y[0, \tau))$ where $Y$ is the simple random walk started at $v_i$ and stopped at $\tau$, the hitting time of $E_{i - 1}$. 
\end{itemize}
Last, define $T = \bigcup_{i = 0}^\infty E_i$. We refer to this procedure as \textbf{Wilson's algorithm rooted at $h$}. The random tree $T$ generated by Wilson's algorithm rooted at $h$ is clearly a spanning tree of $G$; the next lemma shows that it is distributed as the UST of $G$.

\begin{lemma}[Wilson meets Doob]
\label{lem:Wilson}
Let $G=(V,E)$ be an infinite, connected, locally finite, recurrent graph and let $h\in \mathrm{ext}(\c{H})$.
	The tree $T$ generated by Wilson's algorithm rooted at $h$ is distributed as the uniform spanning tree of $G$. In particular, the law of $T$ is independent of the chosen enumeration of $V$ and the choice of $h\in \mathrm{ext}(\c{H})$. 
\end{lemma}

\begin{remark}
It follows by taking convex combinations that the same statement also holds when $h$ is \emph{not} extremal.
\end{remark}

We will deduce Lemma~\ref{lem:Wilson} from the following lemma, which allows us to think of the Doob-transformed walk $\wh{X}^h$  as a limit of conditioned simple random walks on $G$. For the purposes of this lemma we think of our walks as belonging to the space of sequences in $V$ equipped with the product topology.


\begin{lemma}[Local convergence] \label{L: local convergence}
Let $G=(V,E)$ be an infinite, connected, locally finite, recurrent graph and suppose that $z_n$ is a sequence of vertices of $G$ such that $z_n$ converges to $h\in \c{H}$.
If $X$ denotes the random walk on $G$ started at $o$ and $\wh{X}^h$ denotes the Doob-transformation of $X$ as above, then the conditional law of $X$ given that it hits $z_n$ before first returning to $o$ converges weakly to the law of $\wh{X}^h$.
\end{lemma}

\begin{proof}[Proof of Lemma~\ref{L: local convergence}]
	This is a classical result concerning Doob transforms, and can also be deduced from the limit formula \eqref{eq:PK_limit}. We give a brief proof. Let $T_{z_n}$ be the first time the walk hits $z_n$, let $T_o^+$ be the first positive time the walk hits $o$, and let $\varphi = (o, \varphi_1, \ldots, \varphi_m)$ be a path of length $m$ for some $m\geq 1$ with $\varphi_i\neq o$ for every $i>0$. By the Markov property for the simple random walk, 
	\[
		\mathbf{P}_o(X[0, m] = \varphi, T_{z_n} < T_o^+) = 	\mathbf{P}_o(X[0, m] = \varphi)\mathbf{P}_{\varphi_m}(T_{z_n} < T_o),
	\]
and it follows from \eqref{eq:Doob_transform_whole_path} that
	\[
			\mathbf{P}_o(X[0, m] = \varphi, T_{z_n} < T_o^+) = 	\frac{1}{\deg(o)a^h(\varphi_m, o)} \mathbf{P}_o(\wh{X}^h[0, m] = \varphi)\mathbf{P}_{\varphi_m}(T_{z_n} < T_o). 
	\]
	The result follows once multiplying both sides by the effective resistance between $o$ and $z_n$ and using the representation \eqref{eq: PK def} for the potential kernel. 
\end{proof}

	\begin{proof}[Proof of Lemma~\ref{lem:Wilson}]
The standard implementation of Wilson's algorithm rooted at $z_n$ allows us to sample the uniform spanning tree of $G$ in a manner exactly analogous to above, except that we start with a walk run from $o$ until it first hits $z_n$. Now, it is a combinatorial fact that the \emph{loop erasure} of the walk run from $o$ until it first hits $z_n$ does not change its distribution if we condition the walk to hit $z_n$ before returning to $o$: Indeed, the loop-erasure of the entire unconditioned walk is equal to the loop-erasure of the final segment of the walk between its last visit to $o$ and its first visit to $z_n$, and this final segment is distributed as the conditioned walk. Thus, in the standard implementation of Wilson's algorithm, we do not change the distribution of the obtained tree if we condition the first walk to hit $z_n$ before returning to $o$. The claim then follows by taking the limit as $z_n\to\infty$ and using Lemma~\ref{L: local convergence}.
	\end{proof}

	This leads to the following connection between the ends of the UST and the extremal points of the set of harmonic measures from infinity $\c{H}$.

\begin{proposition}
\label{prop:ends_and_H}
Let $G=(V,E)$ be an infinite, connected, locally finite, recurrent graph, let $T$ be the uniform spanning tree of $T$, and let $H$ be a countable subset of $\mathrm{ext}(\c{H})$. Almost surely, for each $h\in H$ there exists an infinite simple path $\Gamma=(\Gamma_1,\Gamma_2,\ldots)$ in $T$ such that 
\[
\frac{a^{h'}(\Gamma_i,x)}{a^{h}(\Gamma_i,x)}\to 0 \qquad \text{ as $i\to\infty$ for each $h' \in H \setminus \{h\}$.}
\]
In particular, $T$ almost surely has at least as many ends as there are extremal points of $\c{H}$.
\end{proposition}

(In the last sentence of this proposition we are not distinguishing between different infinite cardinalities, but merely claiming that if $\c{H}$ has infinitely many extremal points then $T$ has infinitely many ends almost surely.)

\begin{proof}
This is an immediate consequence of Proposition~\ref{prop:identifying_h_from_trace} and Lemma~\ref{lem:Wilson}.
\end{proof}

\noindent \textbf{Orientations.} Let $G=(V,E)$ be an infinite, connected, locally finite, recurrent graph and let $h\in \mathrm{ext}(\c{H})$. When we generate the UST $T$ of $G$ using Wilson's algorithm rooted at $h$, the algorithm also provides a natural \emph{orientation} of $T$, where each edge is oriented in the direction that it is crossed by the loop-erased random walk that contributed that edge to the tree. When $T$ almost surely has the same number of ends as there are extremal points in $\c{H}$, and both numbers are finite (which will always be the case in the unimodular setting by the results of \cite{BvE21}), it follows from Proposition~\ref{prop:ends_and_H} that this orientation is a.s.\ determined by the (unoriented tree): Almost surely, for each $h\in \c{H}$ and $v\in V$ there is a unique infinite ray $(\Gamma_1,\Gamma_2,\ldots)$ starting from $v$ such that
\[
\frac{a^{h'}(\Gamma_i,v)}{a^{h}(\Gamma_i,v)}\to 0 \qquad \text{ as $i\to\infty$ for each $h' \in \mathrm{ext}(\c{H}) \setminus \{h\}$,}
\]
and if we orient the tree in the direction of this ray we must recover the same orientation as if we had generated the oriented tree using Wilson's algorithm rooted at $h$. This fact will play a key role in the proof of our main theorem.

\section{Proof of the main theorem}
\label{sec:main_proof}

\subsection{Reversible and unimodular graphs}
\label{subsec:definitions}
We now give a very brief introduction to unimodular random rooted graphs, referring the reader to \cite{AldousLyonsUnimod2007, CurNotes} for detailed introductions. Let us just recall that $\c{G}_{\bullet, \bullet}$ is the separable metric space of doubly rooted graphs $(G, x, y)$ (modulo graph isomorphisms), equipped with the \textit{local topology}, also known as \textit{Benjamini-Schramm topology}. Similarly defined is the space $\c{G}_{\bullet}$ of rooted graphs $(G, o)$. A \textit{mass transport} is a measurable function $f: \c{G}_{\bullet, \bullet} \to [0, \infty]$. A measure $\PP$ on $\c{G}_{\bullet}$ is called \textit{unimodular} whenever the \emph{mass transport principle}
\[
\wh{\E} \left[\sum_{x \in V} f(G, o, x)\right] = \wh{\E} \left[\sum_{x \in V} f(G, x, o)\right]
\]
holds for all mass transports $f$. A probability measure $\PP$ on $\c{G}_{\bullet}$ is called \textit{reversible} if $(G, o, X_1) \overset{d}{=} (G, X_1, o)$ where $X_1$ is the first step of the simple random walk. The law $\PP$ is called \textit{stationary} if $(G, o) \overset{d}{=} (G, X_1)$ and clearly any reversible graph is stationary. For recurrent graphs, stationarity and reversibility are equivalent \cite{BenjaminiCurien2012}. 

If $\PP$ is the law of a unimodular random graph, with finite expected degree, then biasing it by $\deg(o)$ gives a reversible random graph and whenever $\PP$ is the law of a reversible random graph, then biasing by $\deg(o)^{-1}$ gives a unimodular random graph; see for example \cite{BenjaminiCurien2012}. 

A set $A\subseteq \c{G}_\bullet$ is said to be \textbf{rerooting invariant} if $((g,v)\in A)\Rightarrow ((g,u)\in A)$ for every rooted graph $(g,v) \in \c{G}_\bullet$ and every $u$ in the vertex set of $g$.
A unimodular random rooted graph $(G,o)$ is said to be \textbf{ergodic} if it has probability $0$ or $1$ to belong to any given re-rooting invariant event in $\c{G}_\bullet$. As explain in \cite[Section 4]{AldousLyonsUnimod2007}, this is equivalent to the law of $(G,o)$ being extremal in the weakly compact convex set of unimodular probability measures on $\c{G}_\bullet$. As such, it follows by Choquet theory that every unimodular measure on $\c{G}_\bullet$ may be written as a mixture of ergodic unimodular measures. For our purposes, the upshot of this is that we may assume without loss of generality that $(G,o)$ is ergodic when proving Theorem~\ref{T:main}.

We will also rely on the following characterization of two-ended unimodular random rooted graphs due to Bowen, Kun, and Sabok \cite{bowen2021perfect}, which builds on work of Benjamini and the second author \cite{BenHutch}. Here, a graph $G$ is said to have \textbf{linear volume growth} if for each vertex $v$ of $G$ there exists a constant $C_v$ such that $|B(v,r)|\leq C_v r$ for every $r\geq 1$, where $B(v,r)$ denotes the graph distance ball of radius $r$ around $v$.

\begin{proposition}[\cite{bowen2021perfect}, Proposition 2.1]
\label{prop:two_ended=linear_growth}
Let $(G,o)$ be an infinite unimodular random rooted graph. Then $G$ is two-ended almost surely if and only if it has linear volume growth almost surely.
\end{proposition}

To prove Thorem~\ref{T:main}, it will therefore suffice to prove that if $(G,o)$ is a recurrent unimodular random rooted graph whose UST is two-ended almost surely then $G$ has linear volume growth almost surely.

\subsection{The effective resistance is linear on the spine} \label{s: linear growth}

Let $\PP$ be the joint law of an ergodic recurrent unimodular random rooted graph $(G,o)$ and its uniform spanning tree $T$, which we think of as a triple $(G,o,T)$. It follows by tail triviality of the UST \cite[Theorem 8.3]{BLPS} that the number of ends of $T$ is deterministic conditional on $(G,o)$, and since $(G,o)$ is ergodic that $T$ has some constant number of ends almost surely.
Moreover, it follows from \cite[Theorem 6.2 and Proposition 7.1]{AldousLyonsUnimod2007} that this number of ends is either $1$ or $2$ almost surely, so that $T$ is either one-ended almost surely or two-ended almost surely.

We wish to prove that if $T$ is two-ended almost surely then $G$ is two-ended almost surely also. We will rely on the following theorem of Berestycki and the first author. 

\begin{theorem}[\cite{BvE21}, Theorem 1] \label{T:harm and UST}
	Let $(G, o)$ be a recurrent unimodular random rooted graph with $\E\deg(o)<\infty$. Almost surely, the uniform spanning tree of $G$ is one-ended if and only if the harmonic measure from infinity is uniquely defined.
\end{theorem}

To avoid the unnecessary assumption that $\E \deg(o)<\infty$, we will use the following mild generalization of this theorem, whose proof is given in Appendix~\ref{Appendix:BvE}.

\begin{theorem}\label{T:harm and UST_no_degree_assumption}
	Let $(G, o)$ be a recurrent unimodular random rooted graph. Almost surely, the uniform spanning tree of $G$ is one-ended if and only if the harmonic measure from infinity is uniquely defined.
\end{theorem}

It follows from this theorem together with Proposition~\ref{prop:ends_and_H} that if $T$ is two-ended almost surely then
$|\mathrm{ext}(\c{H})|=2$ almost surely.

Suppose that $T$ is two-ended almost surely and let $\c{S}$ be the \textbf{spine} of $T$, i.e., the unique double-infinite simple path contained in $T$.
We give $T$ an orientation by choosing uniformly at random one of the two ends of $\c{S}$ and directing every edge towards that end, letting the resulting oriented tree be denoted $T^\rightarrow$ with oriented spine $\c{S}^\rightarrow$. 
Since the law of $T^\rightarrow$ is a rerooting-equivariant function of the graph $(G,o)$, the triple $(G,T^\rightarrow,o)$ is unimodular.
Since ``everything that can happen somewhere can happen at the root" \cite[Lemma 2.3]{AldousLyonsUnimod2007} we also have that the origin belongs to $\c{S}$ with positive probability
and hence that we can define a law $\PSpine$ on triplets $(G, T^\rightarrow, o)$ (which we can view as a rooted network) by conditioning $o$ to belong to~$\c{S}$.
%
%
The law $\PSpine$ has the very useful property that it is stationary under shifts along the spine, which we now define. Each vertex $v \in \c{S}$ has a unique oriented edge emanating from it in $\c{S}^\rightarrow$, and we will write $\sigma(v)$ for the vertex on the other end of this edge. The map $v\mapsto \sigma(v)$ can be thought of as a shift, following the orientation along the spine, and there is also a well-defined backwards shift $\sigma^{-1}$ mapping each $x\in \c{S}$ to the unique vertex $v\in\c{S}$ with $\sigma(v)=x$.

\begin{lemma}\label{lem:stationary}
	The law $\PSpine$ is invariant under the shift $\sigma$.  
\end{lemma}

\begin{proof}
	 Let $A$ be any Borel set of triples $(g,t^\rightarrow,v)$ where $(g,v)$ is a rooted graph and $t^\rightarrow$ is an oriented spanning tree of $g$, and define the mass transport
	\[
		f(g,t^\rightarrow, v, w) := \id\left(\text{$t^\rightarrow$ is two-ended, $w$ is in the spine of $t^\rightarrow$, $v=\sigma(w)$, and $(g,t^\rightarrow,w)\in A$}\right).
	\]
	Note that there only exists one vertex $v$ such that $v = \sigma(w)$ and, vice-versa, for each $v$ in the spine of $t^\rightarrow$ there is only one $v$ in the spine of $t^\rightarrow$ such that $\sigma(v) = o$ and $v \in \mathcal{S}$.
	 Therefore, 
	\[
		\sum_{v \in V} f(G,T^\rightarrow, v, o) = \id\left(\text{$T^\rightarrow$ is two-ended, $o$ is in the spine of $T^\rightarrow$, and $(G,T^\rightarrow,o)\in A$}\right)
	\]
	and 
	\[
		\sum_{v \in V} f(G,T^\rightarrow, v, o) = \id\left(\text{$T^\rightarrow$ is two-ended, $o$ is in the spine of $T^\rightarrow$, and $(G,T^\rightarrow,\sigma(o))\in A$}\right)
	\]
	Using the mass-transport principle we thus have that 
	\begin{multline*}
	\PP\left(\text{$T^\rightarrow$ is two-ended, $o$ is in the spine of $T^\rightarrow$, and $(G,T^\rightarrow,o)\in A$}\right) \\= 
	\PP\left(\text{$T^\rightarrow$ is two-ended, $o$ is in the spine of $T^\rightarrow$, and $(G,T^\rightarrow,\sigma(o))\in A$}\right)
	\end{multline*}
	which shows the result because $\PP(o \in \c{S}) > 0$ and $T$ is two-ended a.s. by assumption. 
\end{proof}






The main goal of this section is to show that along the spine of the UST, the effective resistances on the original graph must grow linearly under the assumption that the UST has two ends (and thus a well-defined spine). 
Heuristically, this tells us that if a graph is unimodular and the uniform spanning tree is two-ended, then the actual graph should in some sense be ``close'' to the line $\ZZ$. 


\begin{proposition} \label{P:two-ends->linear} 
	The limit $\lim_{n\to\infty}\frac{1}{n} \Reff(o \leftrightarrow \sigma^n(o))=\lim_{n\to\infty}\frac{1}{n} \Reff(o \leftrightarrow \sigma^{-n}(o))$ exists and is positive
	$\PSpine$-a.s.
\end{proposition}

Note that the existence part of this proposition is an immediate consequence of the subadditive ergodic theorem; the content of the proposition is that the limit is positive.

\medskip

As discussed above, it follows from Proposition~\ref{prop:ends_and_H} and Theorem~\ref{T:harm and UST} that, $\PP_{\c{S}}$-almost surely, there are exactly two extremal elements of $\c{H}$, which we call ``$\ell$'' and ``$r$'', which satisfy
\begin{align}
\label{eq:identifying_ends}
\frac{a^r(\sigma^n(o), v)}{a^\ell(\sigma^n(o), v)} \to \begin{cases} \infty & \text{ as $n\to +\infty$}\\
0 & \text{ as $n\to-\infty$}\end{cases}
\end{align}
for every $v\in V$. (In particular, the random choice of orientation of $T$ we made when defining $\PP_{\c{S}}$ is equivalent to randomly choosing which of the two extremal elements of $\c{H}$ to call ``$r$''.)
%
%
Consider the function $V\to\RR$ defined by
\[
	M_{o}(x) := a^r(x, o)-a^\ell(x, o).
\]
We will show that $M_{o}(\sigma^n(o))$ grows linearly in $n$ and deduce from this that the effective resistance does too. 
The latter fact can be seen using \eqref{eq: PK def}, from which it follows that
\[
	M_{o}(x) = (r_{x, o}(x) - \ell_{x, o}(x)) \Reff(o \leftrightarrow x). 
\]
In the remainder we will slightly abuse notation to write $M_{m}(n) := M_{\sigma^m(o)}(\sigma^n(o))$ for $n, m \in \ZZ$. 
The first main ingredient is that $M_o(n)$ is an \emph{additive cocyle}. 
\begin{lemma}\label{lem:cocycle}
		$M_o(n + m) = M_o(n) + M_{n}(n + m)$ for every $n,m\in \ZZ$.
\end{lemma}

\begin{proof}
	This is a direct consequence of Proposition 3.5 in \cite{BvE21}, stating that
	\[
	a^{\#}(x, o) - a^{\#}(y, o) = a^{\#}(x, y) - \frac{\Gr_{y}(x, o)}{\deg(o)}
	\]
	for each $\# \in \{\ell,r\}$ and all $x, y \in V$. Indeed, it follows from this identity that
	\begin{align*}
	M_{o}(n + m) - M_o(n) &= a^r(\sigma^{n + m}(o),o)-a^\ell(\sigma^{n + m}(o),o) - a^r(\sigma^{n}(o),o) + a^\ell(\sigma^{n}(o),o)\\
	&= \left[a^r(\sigma^{n + m}(o),\sigma^{n}(o))- \frac{\Gr_{\sigma^n(o)}(\sigma^{n+m}(o), o)}{\deg(o)}\right] \\&\hspace{4cm}- \left[a^\ell(\sigma^{n + m}(o),\sigma^{n}(o)) - \frac{\Gr_{\sigma^n(o)}(\sigma^{n+m}(o), o)}{\deg(o)}\right]\\
	&= a^r(\sigma^{n + m}(o),\sigma^{n}(o)) - a^\ell(\sigma^{n + m}(o),\sigma^{n}(o)) =M_{n}(n + m)
	\end{align*}
	for every $n,m\in \Z$ as claimed.
\end{proof}

Let us also make note of the following key property of this additive cocycle. 

\begin{lemma}
\label{lem:eventual_positivity}
$\PP_{\c{S}}$-almost surely, $M_o(n)$ is positive for all sufficiently large positive $n$ and negative for all sufficiently large negative $n$. Moreover, 
\[
M_o(n) \sim a^r(\sigma^n(o), o) = r_{\sigma^n(o), o}(\sigma^n(o)) \Reff(o \leftrightarrow \sigma^n(o))
\]
$\PP_{\c{S}}$-almost surely as $n\to\infty$.
\end{lemma}

\begin{proof}
This follows immediately from \eqref{eq:identifying_ends} and the definition of $M_o(n)$. 
\end{proof}

We will deduce Proposition~\ref{P:two-ends->linear} from Lemma~\ref{lem:eventual_positivity} together with the following general fact about stationary sequences.

\begin{proposition}
\label{prop:ergodic}
Let $(Z_i)_{i\in \ZZ}$ be a stationary sequence of real-valued random variables and suppose that $\sum_{i=0}^n Z_{-i} >0$ for all sufficiently large $n$ almost surely. Then $\limsup_{n\to\infty} \frac{1}{n}\sum_{i=0}^n Z_i>0$ almost surely.
\end{proposition}

\begin{proof}
For each $n\in \ZZ$ let $R_n = \inf \{m \geq 0: \sum_{i=n}^{n+m}Z_i>0 \}$, so that $R_n=0$ whenever $Z_n>0$ and $(R_n)_{n\in \ZZ}$ is a stationary sequence of $\{0,1,\ldots\}$-valued random variables. It follows from the definitions that if $n \leq m$ then either $n+R_n < m$ or $n+R_n \geq m+R_m$, so that the intervals $[n,n+R_n]$ and $[m,m+R_m]$ are either disjoint or ordered by inclusion. On the other hand, we have by stationarity and the hypotheses of the Proposition that for each $n\in \ZZ$ there  almost surely exists $N_n<\infty$ such that $\sum_{i=n-m}^{n-1} Z_{i} >0$ for every $m \geq N_n$ and hence that $R_{n-m}+(n-m) < n$ for every $m \geq  N_n$, so that each $n\in \ZZ$ is contained in at most finitely many of the intervals $[m,m+R_m]$ almost surely. Using the fact that these intervals are either disjoint or ordered by inclusion, it follows that there is a unique decomposition of $\Z$ into maximal intervals of this form 
\[
\ZZ = \bigcup\Bigl\{[k,k+R_k] : k \in \ZZ, [k,k+R_k] \nsubseteq [m,m+R_m] \text{ for every $m\in \ZZ \setminus \{k\}$}\Bigr\}.
\]
Thus, if we define $Y_n$ by
\[
Y_n = \begin{cases}\sum_{i=k}^{n} Z_i & n = k+R_k \text{ for some $k\in \ZZ$ such that $[k,k+R_k]$ maximal} \\ 
0 & \text{ otherwise}
\end{cases}
\]
then $(Y_n)_{n\in \ZZ}$ is a stationary sequence of non-negative random variables such that $Y_n$ is positive whenever $n$ is the right endpoint of a maximal interval. Since $Y_n$ is non-negative and the set of $n$ such that $Y_n \neq 0$ is almost surely non-empty, it follows from the ergodic theorem applied to $(\min\{Y_n,1\})_{n\in\ZZ}$ that
\[
\liminf_{n\to\infty} \frac{1}{n}\sum_{i=0}^n Y_n >0
\]
almost surely. The claim follows since if $-m$ is the left  endpoint of the maximal interval containing $0$ then 
\[\sum_{i=0}^n Y_n = \sum_{i=-m}^n Z_i\]
for every $n$ that is the right  endpoint of some maximal interval.
\end{proof}

\begin{proof}
It follows from Lemma~\ref{lem:stationary} that $(M_n(n+1))_{n\in \ZZ}$ is a stationary sequence under $\PP_{\c{S}}$ and from Lemma~\ref{lem:cocycle} that $M_o(n)=\sum_{i=0}^{n-1}M_i(i+1)$ for every $n\geq 0$ and $M_o(-n)=\sum_{i=-n}^{-1}M_i(i+1)$ for every $n\leq 0$. Thus, Lemma~\ref{lem:eventual_positivity} implies that the stationary sequence $(M_n(n+1))_{n\in \ZZ}$ satisfies the hypotheses of Proposition~\ref{prop:ergodic} and hence that
\[
\limsup_{n\to\infty} \frac{M_o(n)}{n} >0
\]
almost surely. On the other hand, the subadditive ergodic theorem implies that the limit $\lim_{n\to\infty}\frac{1}{n} \Reff(o \leftrightarrow \sigma^n(o))$ exists 
	$\PSpine$-a.s., and since
\[
M_o(n)=\left(r_{\sigma^n(o), o}(\sigma^n(o)) -\ell_{\sigma^n(o), o}(\sigma^n(o)) \right)\Reff(o \leftrightarrow \sigma^n(o)) \leq \Reff(o \leftrightarrow \sigma^n(o))
\]
we must have that
\[\lim_{n\to\infty}\frac{1}{n} \Reff(o \leftrightarrow \sigma^n(o))>0\]
$\PSpine$-a.s.\ as claimed. The fact that the negative-$n$ limit $\lim_{n\to\infty}\frac{1}{n} \Reff(o \leftrightarrow \sigma^{-n}(o))$ also exists and is equal to the positive-$n$ limit a.s.\ follows from the subadditive ergodic theorem.
\end{proof}

\subsection{Completing the proof}

We now complete the proof of the main theorem.

\begin{proof}[Proof of Theorem~\ref{T:main}] 
It suffices by Proposition~\ref{prop:two_ended=linear_growth} to prove that if $(G,o)$ is a recurrent unimodular random rooted graph whose UST is two-ended almost surely then $G$ has linear volume growth almost surely. As before, we write $\c{S}$ for the spine of the oriented UST $T^\rightarrow$, write $\mathbb{P}_{\c{S}}$ for the conditional law of $(G,T^\rightarrow,o)$ given that $o\in \c{S}$, and write $\sigma$ for the shift along the spine as in Lemma~\ref{lem:stationary}. 

%
%
%
	For each $x\in V$ let $S(x)$ be an element of $\c{S}$ of minimal graph distance to $x$, choosing one of the finitely many possibilities uniformly and independently at random for each $x$ where this point is not unique. 
	 Letting $S^{-1}(v)=\{x\in V:S(x)=v\}$ for each $v\in \c{S}$, we have by the mass-transport principle that
	\begin{align*}
	 	\E_{\c{S}}|S^{-1}(o)| = \E\bigl[ |S^{-1}(o)| \mid o \in \c{S}\bigr] &= \mathbb{P}(o\in \c{S})^{-1}\E\left[ \sum_{x\in V} \id(o=\c{S}(x))\right] \\&= \mathbb{P}(o\in \c{S})^{-1}\E\left[ \sum_{x\in V} \id(x=\c{S}(o))\right]=\mathbb{P}(o\in \c{S})^{-1}<\infty.
	\end{align*}
	We thus have a stationary sequence of random variables $(|S^{-1}(\sigma^i(o))|)_{i \in \ZZ}$ with uniformly finite mean, and the ergodic theorem implies that
	\begin{equation}
	\label{eq:ergodic_theorem_S-1}
		\lim_{i\to \infty} \frac{1}{2n}\sum_{i=-n}^n |S^{-1}(\sigma^i(o))| <\infty
	\end{equation}
	almost surely. On the other hand, letting $B(o,r)$ be the graph distance ball of radius $r$ around $o$ for each $r\geq 1$, we have by definition of $S^{-1}$ that
	\begin{equation}
	\label{eq:B_containment}
B(o,r) \subseteq \bigcup \left\{S^{-1}\left(\sigma^n(o)\right):n\in \ZZ, d(o,\sigma^n(o))\leq 2r\right\}
	\end{equation}
	for each $r\geq 1$. Proposition~\ref{P:two-ends->linear} together with the trivial inequality $\Reff(x\leftrightarrow y)\leq d(x,y)$ imply that there exists a positive constant $c>0$ such that $d(o,\sigma^n(o))\geq c|n|$ for all sufficiently large (positive or negative) $n$ almost surely, and together with \eqref{eq:ergodic_theorem_S-1} and \eqref{eq:B_containment} this implies that $\limsup_{r\to \infty} \frac{1}{r}|B(o,r)|<\infty$ almost surely. This completes the proof. 
\end{proof}

\bibliography{aldous-lyons.bib}

\appendix

\section{Uniqueness of the potential kernel implies one-endedness of the UST, without finite expected degree}

\label{Appendix:BvE}

In this appendix we prove Theorem~\ref{T:harm and UST_no_degree_assumption}, which generalizes the theorem of Berestycki and the first author concerning the equivalence of the UST being one-ended and uniqueness of the harmonic measure from infinity to the case that the unimodular random rooted graph does not necessarily have finite expected degree. A secondary purpose of this appendix is to give a brief and self-contained account of those results of \cite{BvE21} that are needed for our main results.  Since recurrent graphs whose USTs are one-ended always have unique harmonic measure from infinity \cite[Theorem 14.2]{BLPS}, it suffices to prove that the converse holds under the additional assumption of unimodularity.
Moreover, it suffices as usual to consider the case that $(G,o)$ is ergodic.

\medskip

Suppose that $(G,o)$ is an ergodic recurrent unimodular random rooted graph for which $\c{H}$ is a singleton almost surely. We write $h$ for the unique element of $\c{H}$ and $a$ for the corresponding potential kernel. For each $c>0$ consider the event $A_c=\{\limsup_{x\to \infty}h_{x,o}(x)\geq c\}=\{$for each $\eps>0$ there exist infinitely many vertices $x$ with $h_{x,o}(x)\geq c-\eps\}$. 
As explained in detail in \cite[Lemma 5.3]{BvE21} (which concerns deterministic recurrent graphs),
we have that 
\begin{equation}
\label{eq:h_asymptotic}
h_{x,o}(x)\sim h_{x,w}(x) \qquad \text{ as $x\to \infty$ for each fixed $w\in V$},
\end{equation}
which 
 implies that $A_c$ is re-rooting invariant. Since $(G,o)$ was assumed to be ergodic we deduce the following.

\begin{lemma}
Let $(G,o)$ be an ergodic unimodular random rooted graph. If $G$ is almost surely recurrent with a uniquely defined harmonic measure from infinity then the event $A_c$ has probability $0$ or $1$ for each $c\in (0,1)$.
\end{lemma}

The next lemma is proven in \cite{BvE21} using an argument that relies on reversibility (and hence on the assumption $\E \deg(o)<\infty$). We give an alternative proof using F{\o}lner sequences that works without this assumption.

\begin{lemma}
\label{lem:A_1/2}
Let $(G,o)$ be an ergodic unimodular random rooted graph. If $G$ is almost surely recurrent with a uniquely defined harmonic measure from infinity then the event $A_{1/2}$ holds almost surely.
\end{lemma}

\begin{proof}
It suffices to prove that $A_{c}$ holds with positive probability for every $c<1/2$.
Since $(G,o)$ is recurrent, it follows from the results of \cite[\S 8]{AldousLyonsUnimod2007} that $(G,o)$ is \emph{hyperfinite}, meaning that there exists a sequence of random subsets $(\omega_n)_{n\geq 1}$ of $E$ such that 
\begin{enumerate}
	\item Every component of the subgraph spanned by $\omega_n$ is finite almost surely for each $n\geq 1$.
	\item $\omega_n \subseteq \omega_{n+1}$ for each $n\geq 1$ and $\bigcup_{n\geq 1} \omega_n = E$.
	\item The random rooted edge-labelled graph $(G,o,(\omega_n)_{n\geq 1})$ is unimodular.
\end{enumerate}
Let $n\geq 1$ and let $K_n$ be the component of $o$ in $\omega_n$. Then we have by the mass-transport principle that
\[
\E\left[\frac{1}{|K_n|}\sum_{x\in K_n}\id\left(h_{x,o}(x)\geq \frac{1}{2}\right)\right]
=
\E\left[\frac{1}{|K_n|}\sum_{x\in K_n}\id\left(h_{x,o}(o)\geq \frac{1}{2}\right)\right],
\]
and since the sum of the two sides is at least $1$ it follows that
\[
\E\left[\frac{1}{|K_n|}\sum_{x\in K_n}\id\left(h_{x,o}(x)\geq \frac{1}{2}\right)\right] \geq \frac{1}{2}
\]
and hence by Markov's inequality that
\[
\mathbb{P}\left(\left|\{x\in K_n : h_{x,o}(x)\geq \tfrac{1}{2}\}\right| \geq \tfrac{1}{4}|K_n| \right) \geq 1- \frac{4}{3} \E\left[\frac{1}{|K_n|}\sum_{x\in K_n}\id\left(h_{x,o}(x)< \frac{1}{2}\right)\right] \geq \frac{1}{3}.
\]
Since $|K_n|\to \infty$ almost surely as $n\to\infty$, it follows from this and Fatou's lemma that
\[
\mathbb{P}(A_{1/2})\geq \mathbb{P}\left(\left|\{x\in K_n : h_{x,o}(x)\geq \tfrac{1}{2}\}\right| \geq \tfrac{1}{4}|K_n| \text{ for infinitely many $n$}\right) \geq \frac{1}{3}
\]
and hence by ergodicity that $\mathbb{P}(A_{1/2})=1$ as claimed.
\end{proof}




\begin{lemma}
\label{lem:hit_sets_with_h_lower_bound}
Let $G=(V,E)$ be an infinite, connected, locally finite recurrent graph with uniquely defined harmonic measure from infinity $h$, let $o\in V$ and let $a$ be the associated potential kernel. If $A$ is any infinite set of vertices with $\inf_{x\in A}h_{x,o}(x)>0$, the Doob-transformed walk $\wh{X}$ visits $A$ infinitely often almost surely.
\end{lemma}

\begin{proof}
We have by \eqref{eq:hitting prop CRW} that $\wh{\mathbf{P}}_o(\wh{X}$ hits $x)=h_{o,x}(x)$ for every $x\in V$, and it follows by Fatou's lemma that $\wh{\mathbf{P}}($hit $A$ infinitely often$)\geq \inf_{x\in A}h_{x,o}(x)>0$. On the other hand, we have by Theorem~\ref{thm:Liouville} and the assumption that $h$ is unique that $\wh{X}$ has trivial tail $\sigma$-algebra, so that $\wh{\mathbf{P}}($hit $A$ infinitely often$)=1$ as claimed.
\end{proof}

\begin{proposition}
Let $G=(V,E)$ be an infinite, connected, locally finite recurrent graph with uniquely defined harmonic measure from infinity $h$, let $o\in V$, let $a$ be the associated potential kernel, and suppose that $\liminf_{x\to\infty} h_{x,o}(x)>0$.
If $\wh{X}$ and $\wh{Y}$ are independent copies of the Doob-transformed walk started at some vertices $x$ and $y$, then $\{\wh{X}_n:n\geq 0\}\cap \{\wh{Y}_n:n \geq 0\}$ is infinite almost surely.
\end{proposition}

\begin{proof}
Let $\delta>0$ be such that $A=\{x\in V: h_{x,o}(x)\geq \delta\}$ is infinite. Applying Lemma~\ref{lem:hit_sets_with_h_lower_bound} yields that $\wh{X}\cap A$ is infinite almost surely, and applying Lemma~\ref{lem:hit_sets_with_h_lower_bound} a second time yields that $\wh{Y}\cap \wh{X} \cap A$ is infinite almost surely.
\end{proof}

Applying this proposition together with the results of \cite{Lyons_2003}, which imply that an independent Markov process and loop-erased Markov process intersect infinitely almost surely whenever the corresponding two independent Markov processes do, we deduce the following immediate corollary.

\begin{corollary}
\label{cor:intersection_property2}
Let $G=(V,E)$ be an infinite, connected, locally finite recurrent graph with uniquely defined harmonic measure from infinity $h$, let $o\in V$, let $a$ be the associated potential kernel, and suppose that $\liminf_{x\to\infty} h_{x,o}(x)>0$.
If $\wh{X}$ and $\wh{Y}$ are independent copies of the Doob-transformed walk started at some vertices $x$ and $y$, then $\{\wh{X}_n:n\geq 0\}\cap \{\mathrm{LE}(\wh{Y})_n:n \geq 0\}$ is infinite almost surely.
\end{corollary}

\begin{proposition}
\label{prop:tip_escape}
Let $G=(V,E)$ be an infinite, connected, locally finite recurrent graph with uniquely defined harmonic measure from infinity $h$, let $o\in V$, let $a$ be the associated potential kernel, and suppose that $\liminf_{x\to\infty} h_{x,o}(x)>0$. For each $x\in V$, let $X$ be a random walk started at $x$ and let $\wh{Y}$ be a Doob-transformed walk started at $o$. Then
\[
\lim_{x\to \infty}\mathbb{P}\left(\{X_n:0\leq n \leq T_o\} \cap \{\mathrm{LE}(\wh{Y})_m: m \geq 0\}=\{o\}\right)=0.
\]
\end{proposition}

\begin{proof}
As $x\to \infty$, the law of the time-reversed final segment $(X_{T_o},X_{T_o-1},\ldots,X_{T_o-k})$ converges to that of $(\wh{X}_0,\ldots,\wh{X}_k)$ for each $k\geq 1$, and the claim follows from Corollary~\ref{cor:intersection_property2}.
\end{proof}

\begin{proof}[Proof of Theorem~\ref{T:harm and UST_no_degree_assumption}]
The fact that $G$ has a unique harmonic measure from infinity means that we can endow the uniform spanning tree of $G$ with an orientation in a canonical way: Suppose that we exhuast $V$ by finite sets $V=\bigcup V_n$ and let $G_n^*$ be defined by contracting $V\setminus V_n$ into a single boundary vertex $\partial_n$, so that the UST of $G$ can be expressed as the weak limit of the USTs of the graphs $G_n^*$. If for each $n\geq 1$ we orient the UST of $G_n^*$ towards the boundary vertex $\partial_n$ to obtain an oriented tree $T_n^\rightarrow$, then the uniqueness of the harmonic measure from infinity on $G$ implies that the law of $T_n^\rightarrow$ converges weakly to the law of an oriented spanning tree $T^\rightarrow$ of $G$, which can be thought of as a canonical (but potentially random) orientation of the UST of $G$. Indeed, if we fix an enumeration $v_1,v_2,\ldots$ of $V$ with $v_1=o$ we can sample $T_n^\rightarrow$ using Wilson's algorithm rooted at $\partial_n$, starting with the vertices in the order they appear in the enumeration of $V$, and orienting the edges of the tree in the direction they are crossed by the loop-erased walk that contributed them to the tree. In the infinite-volume limit (since only the part of the first walk after its final visit to $o$ contributes to its loop erasure), this corresponds to doing Wilson's algorithm where the first walk started at $o$ is Doob-transformed and the remaining walks are ordinary simply random walks.

An important consequence of this discussion is that if we sample the oriented uniform spanning tree using Wilson's algorithm rooted at infinity, where the first random walk is a Doob-transformed walk started at $o$ and the remaining walks are ordinary simple random walks, the distribution of the resulting oriented tree $T^\rightarrow$ does not depend on the choice of the root vertex $o$, since it is given by the limit of the USTs of $G_n^*$ oriented towards $\partial_n$ independently of the choice of exhaustion. Given the oriented tree $T^\rightarrow$, we say that a vertex $u$ is in the \textbf{future} of a vertex $v$ if the unique infinite oriented path emanating from $v$ passes through $v$, and say that $u$ is in the \textbf{past} of $v$ if $v$ is in the future of $u$. 

Let $(\omega_n)_{n\geq1}$ be a sequence witnessing the fact that $(G,o)$ is hyperfinite as in the proof of Lemma~\ref{lem:A_1/2} and let $K_n$ be the cluster of $o$ in $\omega_n$ for each $n\geq 1$. We have by the mass-transport principle that
\[
\E\left[\frac{1}{|K_n|}\sum_{x\in K_n}\id\left(\text{$x$ in past of $o$}\right)\right]
=
\E\left[\frac{1}{|K_n|}\sum_{x\in K_n}\id\left(\text{$x$ in future of $o$}\right)\right].
\]
On the other hand, letting $\c{S}$ be the set of vertices belonging to a doubly infinite path in $T$, we also have that
\[
\E\left[\frac{1}{|K_n|}\sum_{x\in K_n}\id\left(\text{$x$ in past or future of $o$}\right)\right] \geq \E\left[\frac{1}{|K_n|}\sum_{x\in K_n} \id(\text{$o,x\in \c{S}$})\right]
\]
and we can use the mass-transport principle again to bound
\begin{align*}
\E\left[\frac{1}{|K_n|}\sum_{x\in K_n} \id(o,x\in \c{S})\right] &=\E\left[\frac{1}{|K_n|^2}\sum_{x,y\in K_n} \id(o,x\in \c{S})\right]=\E\left[\frac{1}{|K_n|^2}\sum_{x,y\in K_n} \id(x,y\in\c{S})\right]
\\&= \E\left[\left(\frac{|K_n \cap \c{S}|}{|K_n|}\right)^2\right] \geq \E\left[\frac{|K_n \cap \c{S}|}{|K_n|}\right]^2 = \mathbb{P}(o\in \c{S})^2.
\end{align*}
Putting these two estimates together, it follows that
\begin{equation}
\label{eq:past_and_S}
\E\left[\frac{1}{|K_n|}\sum_{x\in K_n}\id\left(\text{$x$ in past of $o$}\right)\right] \geq \frac{1}{2}\mathbb{P}(o\in \c{S})^2.
\end{equation}
On the other hand, if we sample $T^\rightarrow$ using Wilson's algorithm rooted at infinity, starting with a Doob-transformed $\wh{Y}$ started at $o$ followed by an ordinary random walk $X$ started at $x$, the vertex $x$ belongs to the past of $o$ if and only if the walk $X$ first hits the loop-erasure of $\wh{Y}$ at the vertex $o$. Proposition~\ref{prop:tip_escape} implies that this probability tends to zero as $x\to\infty$,and it follows by bounded convergence that
\begin{equation}
\label{eq:density_zero_past}
\E\left[\frac{1}{|K_n|}\sum_{x\in K_n}\id\left(\text{$x$ in past of $o$}\right)\right]\to 0
\end{equation}
as $n\to\infty$. Putting together \eqref{eq:past_and_S} and \eqref{eq:density_zero_past} yields that $\mathbb{P}(o\in \c{S})=0$.  Since ``everything that can happen somewhere can happen at the root"  \cite[Lemma 2.3]{AldousLyonsUnimod2007}, it follows that $\c{S}=\emptyset$ almost surely and hence that $T$ is one-ended almost surely as claimed.
\end{proof}

\end{document}